\documentclass[]{article}
\usepackage{graphicx}
\usepackage{color}
\usepackage{amssymb,amsmath,epsfig}
\usepackage{color}
\usepackage{graphicx}
\usepackage{amsfonts,amssymb,latexsym}
\usepackage{amsmath}
\usepackage{fancybox}
\usepackage{amsthm}

\DeclareMathOperator{\trace}{Trace}

\newcommand{\R}{\mathbb R}

\newcommand\E{\mathbb E}

\newcommand\bP{\mathbb P}

\newtheorem*{thm*}{Theorem}
\newtheorem*{prop*}{Proposition}
\newtheorem{lem}{Lemma}

\newtheorem{defn}{Definition}

\newtheorem{prop}{Proposition}

\newtheorem{rem}{Remark}
\newcounter{rea}
\newcounter{rek}
\newtheorem{thm}{Theorem}

\begin{document}
\begin{center}
{\large {\bf  Random Discretization of the Finite Fourier Transform and Related Kernel
Random Matrices.}}\\
\vskip 1cm Aline Bonami$^a$ and Abderrazek Karoui$^b$ {\footnote{
This work was supported in part by the  French-Tunisian  CMCU project 10G 1504 project and the
Tunisian DGRST  research grants  UR 13ES47.}}
\end{center}

\vskip 0.5cm {\small
\noindent
$^a$ Institut  Denis-Poisson, Universit\'e d'Orl\'eans, Coll\'egium Sciences et Techniques, B\^atiment de math\'ematiques - Route de Chartres, B.P. 6759 - 45067 Orl\'eans cedex 2, FRANCE.\\
\noindent $^b$ University of Carthage,
Department of Mathematics, Faculty of Sciences of Bizerte, Jarzouna 7021,  TUNISIA.}\\
Emails: aline.bonami@univ-orleans.fr ( A. Bonami), abderrazek.karoui@fsb.rnu.tn (A. Karoui)\\

\noindent{\bf Abstract}---  This paper is centred on the spectral study of a Random Fourier matrix, that is an $n\times n$ matrix $A$ whose $(j, k)$ entries are $\exp(2i\pi m X_jY_k)$, with $X_j$ and $Y_k$ two i.i.d sequences of random variables and $1\leq m\leq n$ is a real number. When they are uniformly distributed on a symmetric interval, this may be seen as a random discretization of the Finite Fourier transform, whose spectrum has been  extensively  studied in relation with band-limited functions. Moreover, this particular case of random Fourier matrix has been proposed in wireless telecommunication in order to approach the singular values of some channel matrices.

Our study is two-fold. Firstly, by pushing forward concentration inequalities, we find an accurate comparison in $\ell^2$- norm between the spectrum of $A^*A$ and the one of an integral operator that can be defined in terms of the two probability laws chosen for the rows and the columns. Our study includes the one of stationary Hermitian kernel matrices and can be generalized to non stationary ones, for which the same kind of comparison with an integral operator is possible. Because of  possible applications in the data science area, these last  matrices have been largely studied in the literature and our results are compared with previous ones.

Secondly we concentrate on uniform distributions for the laws of $X_j$'s and $Y_k$'s, for which the integral operator is the well-known Sinc-kernel operator with parameter $m.$  Our previous study allows to translate to random Fourier matrices the knowledge that we have on the spectrum of this operator. We have for them asymptotic results for $m, n$ and $n/m$ tending to $\infty$, as well as   non asymptotic bounds in the spirit of recent work on the integral operators. As an application, we give fairly good approximations of the number of degrees of freedom and the capacity of  a MIMO  wireless communication network approximation model. The $\ell^2-$estimates given in the first part of this work seem  the right tool  to transfer approximations of these quantities from the integral operator to the random matrix.

  Finally, we provide the reader with some numerical examples that illustrate the theoretical results of this paper.\\ 



\section{Introduction.}
In this work, we are interested in  the spectra of
two families of $n\times n$ random matrices with real or  complex coefficients, namely matrices $A^*A$  and  $H_{\kappa},$ where the entries of $A$ and $H_{\kappa}$ are given by 
\begin{equation}\label{coefficients}
A_{j,k} = \frac {\sqrt m}{n}\exp(2i\pi m Z_jY_k),\quad H_{j,k}= \frac{1}{n} \kappa(Y_j-Y_k),\quad 1\leq j,k\leq n.
\end{equation}
Here, $m$ is a positive number,  $1\leq m \leq n$ and the  $Y_j, Z_k,\, j,k=1,\ldots,n$ are $2n$ independent random variables. We assume that the $Y_j$'s follow the same probability law $P$, while the $Z_k$'s follow the law $Q$.
 The two matrices are related by the fact that
\begin{equation}\label{kappa}
\kappa (x)= m\int e^{2i\pi m x y} dQ(y),
 \end{equation}
 which guarantees that $\kappa(x-y)$ is a Hermitian positive definite kernel on $\R.$ It is  real and symmetric whenever the law $Q$ is symmetric.
Moreover, from Bochner's Theorem, all continuous  positive definite functions on $\R$,  with $\kappa(0)=m,$ can be obtained in this way. With this relation between the two random matrices we prove that the two spectra are close when $n$ is large.
Also, we note that the matrices $A^*A$ and $H_{\kappa},$ are related to each other by the  relation
$ \E_Z (A^*A) = H_{\kappa}.$ 
 A typical example for the probability laws $P$ and $Q$ is  the  uniform law on $(-1/2,+1/2).$ We will be able to go much further in the study of the spectra in this case, but it is  interesting to see that an important part of our   study  does not depend on the choice of these probability laws. 
 The matrix $A$ may be seen as the matrix of a random  discretization of the  Fourier transform  $\mathcal F_P,$  defined    by
\begin{equation}\label{FiniteFourier}
\mathcal F_P f(x)= \sqrt{m} \, \int e^{2i\pi m xy}f(y) dP(y),
\end{equation}
which we consider as a bounded operator from $L^2(P)$ to $L^2(Q).$ When $P$ is the uniform law on $I=(-1/2, + 1/2),$ this operator is the usual  finite Fourier transform, defined by
\begin{equation}\label{Finite_Fourier}
\mathcal F_m f(x)= \sqrt{m} \, \int_{-1/2}^{1/2} e^{2i\pi m xy}f(y) d y.
\end{equation}
So in this case the matrix $A$ appears as giving a random discretization of the finite Fourier transform.

The product $T_{\kappa}=\mathcal F^*_P\mathcal F_P$ is easily computed and found to have $\kappa$ as kernel, so that
 \begin{equation}\label{kappam1}
T_{\kappa}(f)(x)=  \int \kappa(x,y) f(y)\, dP(y),\quad \kappa(x,y)= \kappa(x-y),
\end{equation}
where $\kappa(\cdot)$ is as given by \eqref{kappa}. Similarly, the matrix $H_{\kappa}$ may be seen as a random discretization of the  operator  $T_{\kappa}$. So $\mathcal F_P$ can be considered as a continuous analogue of $A$ while $T_\kappa$ is a continuous analogue of $H_\kappa.$

With the normalization constant $\frac {\sqrt m}{n}$ given in  the coefficients $A_{j,k},$  the Hilbert-Schmidt norm  of $A$ is equal to the Hilbert-Schmidt norm of   $\mathcal F_P.$ Recall that the Hilbert-Schmidt norm of $A$ is given by
${\displaystyle  \| A\|_{HS}^2 =\sum_{j,k=1}^n |A_{j,k}|^2.}$  The traces of $A^*A, H_\kappa$ and $T_\kappa$ are all equal to $m.$ 
 The operator  $T_{\kappa}$ is a positive definite Hilbert-Schmidt operator, which implies that its spectrum is infinite and countable, eventually completed by zeros, when it is of finite rank. 
We denote by  $\lambda(T_{\kappa})=(\lambda_j(T_{\kappa}))_{j\geq 0},$ the  sequence of its eigenvalues, arranged in decreasing order.  So, the spectra of $A^*A$ and $H_{\kappa}$  are also considered as infinite non decreasing sequences   by giving to $\lambda_j(A^*A)$ (resp. $\lambda_j(H_\kappa)$) the value $0$ for $j\geq n.$ The distance between  spectra will be computed in the usual $\ell^2$-norm.

The whole paper is based on a comparison  of  the spectra  of the matrices  $A^*A$ and $H_{\kappa}$ and  $T_{\kappa}$. In the first part of this work, we prove the following estimates in terms of quadratic means: 
 \begin{equation}\label{GineA2}
  \E\left(\|\lambda(A^*A)-\lambda(T_{\kappa})\|^2_{\ell^2}\right)\leq \frac{2m^2}{n},\quad 
   \E\left(\|\lambda(H_{\kappa})-\lambda(T_{\kappa})\|^2_{\ell^2}\right)\leq \frac{m^2}{n},
   \quad \E\left(\|\lambda(A^*A)-\lambda(H_{\kappa})\|^2_{\ell^2}\right)\leq \frac{m^2}{n}.
  \end{equation}
  Let us make two observations: the first inequality is not a superposition of the two other ones. Also, one may ask whether these inequalities are critical. We cannot answer this question up to now,  but numerical tests show that this is the case up to a small multiplicative constant.
Next, by an elaborated use of the scalar McDiarmid's concentration inequality, we prove that 
the approximation errors in the $\ell^2$-norm are  obtained
 with  large probability. More precisely, for any $\xi > 0,$ we  have 
\begin{equation}\label{Error1}
\|\lambda(A^*A)-\lambda(T_{\kappa})\|_{\ell^2} \leq \frac{\sqrt 2 m(\xi+1)}{\sqrt n},\;
\|\lambda(H_{\kappa})-\lambda(T_{\kappa})\|_{\ell^2} \leq \frac{m(\xi+1)}{\sqrt n},\; \|\lambda(A^*A)-\lambda(H_{\kappa})\|_{\ell^2} \leq\frac{m(\xi+1)}{\sqrt n},
\end{equation}
each of them with probability at least  $1-e^{-\xi^2}.$
Up to our knowledge, the results implying the random Fourier matrix $A$ are entirely new. We discovered the subject in the papers of Desgroseilliers, L\'ev\^eque and Preissmann \cite{DLP, DLP2}. Nonetheless, our comparison of the spectra of  $H_\kappa$ and $T_\kappa$ can be compared with existing results and gives some improvements of known estimates from the literature, in particular those of \cite{Rosasco2, Rosasco1}.  Our proofs for this case have benefited from the papers on random Gram matrices, and in particular  the references \cite{ BBZ, S-TC}. Indeed, the matrix $H_{\kappa}$  is a special case of a more general (kernel) Gram matrix $H_{\scriptscriptstyle \rm K},$ given by 
\begin{equation}\label{Hkappa}
H_{\kappa}=\frac 1n \left[{\kappa} (Y_j, Y_k)\right]_{1\leq j,k\leq n},
\end{equation}
 where the $Y_j$ are sample points drawn randomly according to a probability  law $P$ on some input space $\mathcal X,$ which is assumed to be a locally compact metric space. Also, we assume that the kernel 
   ${\kappa}(\cdot,\cdot)$ is Hermitian,  continuous  on  $\mathcal X  \times \mathcal X$, and   
   positive semi-definite. 
 Under these conditions, the integral
 operator $T_{\kappa},$  with kernel ${\kappa},$  defined on $L^2(P),$ by
 \begin{equation}
 \label{Integral_operator}
 T_{\kappa} (f)(x)= \int_{\mathcal X} {\kappa}(x,y) f(y) \, d P,\quad x\in \mathcal X,
 \end{equation}
is a Hilbert-Schmidt operator. Moreover, if we assume that ${\kappa}(\cdot,\cdot)$ is bounded and if $\displaystyle R=\sup_{y\in \mathcal X} {\kappa}(y,y),$ then our estimates can be generalized in this context as
\begin{equation}\label{generalkernel}
 \E\left(\|\lambda(H_{\kappa})-\lambda( T_{\kappa})\|_{\ell^2}^2\right)\leq \frac{R^2}{n},\quad 
 \|\lambda(H_{\kappa})-\lambda( T_{\kappa})\|_{\ell^2}\leq \frac{R(\xi+1)}{\sqrt n},
\end{equation}
with probability at least $1-e^{-\xi^2}$ for all $\xi> 0$. Since $R=m$ in our particular case, then this is clearly a generalization. We will explain in the next section the nature of improvements obtained in this generalized context compared to \cite{Rosasco2,Rosasco1}. The same proof, which we consider as simpler since it does not appeal to vector-valued concentration inequalities,  is operational in the three situations that we have in view.

We will also see that we have a slight improvement when using the $\ell^\infty$-norm instead of the $\ell^2$ one and find that, with probability $1-e^{-\xi ^2},$ 
$$\|\lambda(H_{\kappa})-\lambda( T_{\kappa})\|_{\ell^\infty}\leq \frac{R(\frac{\xi}{\sqrt 2}+1)}{\sqrt n}.$$
This implies the inequalities given in \cite{S-TCK}, but this is  stronger  since they only consider separate eigenvalues.

Random Gram matrices have been largely studied in connection with big data, machine learning and non linear Principal Component Analysis. This is the motivation of the authors that we cited above. In \cite{ BBZ, S-TC}, one can find the equivalent of \eqref{generalkernel} for each eigenvalue separately. They have also provided inequalities involving the sums of the  first $j$-th eigenvalues for the Gram matrix and  the similar  sums for the integral operator. These inequalities  constitute  one of our tools to prove some results of this work.  

 Such inequalities as \eqref{generalkernel} may be also  used to approximate the spectrum of  integral operators. This is the point of view of V. Koltchinskii and E. Gin\'e \cite{KG}, which is extended to propose MCMC approximation in \cite{Adamczak}. 
 Note that the distance in  $\ell^2$-norm between spectra is already considered in \cite{KG} and this paper has been largely used in subsequent literature.  It follows from our study that the eigenvalues of the integral operator $T_\kappa$ can also be approximated by the squares of the singular values of the matrix $A$. This may be found useful to approximate numerically  eigenvalues of an integral  operator with stationary kernel.

\smallskip
 
Let us come back to our own motivations and applications. The random matrix $\frac {n}{\sqrt m}A$ was proposed by Desgroseilliers, L\'ev\^eque and Preissmann \cite{DLP, DLP2} as   an approximate model for studying the singular values of the channel fading matrix  (after some re-normalization) in a wireless
 communication MIMO (Multi Input Multi Output) transmission network, from $n$ transmitters to $n$ receptors, separated by a large distance. They have  considered two main  issues: firstly the number of degrees of freedom of the system,  secondly  its information capacity.
Generally speaking, the degree of freedom of a MIMO system refers to the number of significant eigenvalues of the associated matrix model. Different definitions have been given in this context in \cite{ST}.  In the present work, for  a   positive semi-definite self-adjoint 
Hilbert-Schmidt  operator $T$ with  a discrete set of non-negative  eigenvalues  $\{\lambda_{k}(T),\,\, k\in \mathbb N\},$  we define the (uniform) number of   degrees of freedom at level $\varepsilon$ by 
   \begin{equation}\label{Deg_infty}
   \deg_\infty (T, \varepsilon)=\min \{k ; \lambda_{k}(T)\leq \varepsilon\}.
   \end{equation} 
A similar definition is used for the degrees of freedom of a positive semi-definite Hermitian  matrix. \\

\noindent 
In the second part of this work, we restrict ourselves to the case where the probability laws $P$ and $Q$ are both given by the uniform law on
$I=(-1/2,+1/2).$ These are the laws used in  \cite{DLP, DLP2} in view of their application to wireless networks. In this case, the kernel $\kappa(x,y)$ is the usual Sinc-kernel with parameter $m,$ given by ${\displaystyle \kappa(x,y)=\frac{\sin(m\pi (x-y))}{\pi(x-y)}}.$ In  this special case, we keep the notation for the matrix $A^*A,$ but the associated matrix $H_{\kappa}$ and the integral operator are now denoted by $H_m$ and $\mathcal Q_m,$ respectively. We are interested here in large values of $n$ and $m$ and this is why we put $m$ as an index. The behaviour of the operator $\mathcal Q_m,$ known as the Sinc-kernel operator,   has been largely explored in the literature, see for example \cite{Bonami-Karoui3, Landau2, Landau,  Slepian2, Osipov}. An easy computation proves that the Hilbert-Schmidt norm of $Q_m$ has order $\sqrt m$. So, when approaching the spectra of $A^*A$ and $H_m$, while errors given in \eqref{GineA2} and \eqref{Error1} are of order $\frac{m}{\sqrt n}$, the relative errors have order $\sqrt{\frac mn}.$ This justifies that we describe the spectra of $A^*A$ and $H_m$ when $m$ and $\frac nm$ tend  to $\infty.$

As a consequence of these approximations, we get 
 the main results of the second part of this work. In particular, based on our non-asymptotic study of the spectrum of the operator 
 $\mathcal Q_m,$ given in \cite{BJK}, we give similar non-asymptotic results for the matrices $A^*A$ and $H_m.$ Also, we give new results
concerning  the degrees of freedom and the capacity, associated with these matrices. More precisely, under the assumption that  $m\geq 4,$ we give the following fairly precise estimates for the degrees of freedom of the random Fourier matrix $A^*A.$ For $\varepsilon >0$ and $0<\alpha <1,$ we have with probability at least $\alpha,$ 
\begin{equation}\label{Degfreedom}
\deg_\infty (A^*A, \varepsilon)=m(1+\mathcal E_A),
\end{equation}
with $|\mathcal E_A| \leq C_{\varepsilon, \alpha}\left( \frac mn +\frac{\log m}{m}\right)$ and $C_\varepsilon$ 
 a constant depending only on  $\varepsilon$ and $\alpha.$ So the number of degrees of freedom is equivalent to $m$ when $m$ and $n/m$ tend to 
 $\infty.$
 
Next, we consider the capacity for the transmission system introduced in \cite{DLP} which has 
${\displaystyle \frac{\sqrt{m}}{n} A}$ as a channel matrix. Now, for $p$ the total power of the transmitters that is assumed to be equally distributed and under our normalization \eqref{coefficients}, we define the  network capacity as
 \begin{equation}
 \label{Capacity}
 C(p)= \log \det \left(I_n +  \frac{n p}{m}  A^*A\right),
 \end{equation}
 where $I_n$ is the identity matrix.  We give the following estimate,
 valid for  $s>0,$
\begin{equation}\label{Capacity1}
\E(C_{A^*A}(s))=\E \sum_{k\geq 0} \log\Big(1+s \lambda_k(A^*A)\Big)= m \log\left(1+s\right)(1+\mathcal E),
 \end{equation}
 with  $|\mathcal E|\leq \delta \left(\sqrt{\frac{m}{n}}+\frac{\log m}{m}\right)(1+\log_+(s)).$ Here $\delta$ is a uniform constant and 
 $x_+=\max(0,x).$ 
   This expectation of the capacity is arbitrarily close to the quantity that one gets when replacing the $m$ first eigenvalues by $1$ and the next ones by $0$, under the condition that $ n/m$ and $m$ are large enough.  This supports the intuition of   \cite{DLP, DLP2} where estimates for $C(n)$ are given under the assumption that
$m= n^\gamma,$ with $1/2<\gamma <1.$ They prove that there exist two positive constants $K, K'$ such that  
 \begin{equation}\label{Capacity2}
K\frac{m}{\log(n)} \leq C(n)\leq  K'm\log(n),
  \end{equation}
with  high probability as $n$ gets large.  It follows from \eqref{Capacity1} that \eqref{Capacity2} can be replaced by 
\begin{equation}\label{Capacity3}
C(n)=(2-\gamma) n^{\gamma} \log(n) \big(1+ \mathcal E\big),\qquad |\mathcal E|=O\big( n^{\frac{\gamma-1}{2}}  \log(n)\big).
\end{equation}
This represents a real improvement, compared to 
the evaluation of the capacity, given in \cite{DLP, DLP2}, when $n$ is large enough. 

 Also, we should mention   that the estimates of the degrees of freedom of the  Hermitian matrix $A^*A,$  given by \eqref{Degfreedom}, are still valid if $A^*A$  is  substituted with  the real and symmetric Sinc-kernel matrix $H_m.$ Hence,  one concludes that as for  the random matrix $A^*A,$ the matrix $H_m$ has  approximately $m$ significant eigenvalues. \\

\smallskip
\noindent
 This work is organized as follows. In section 2, we prove the mean approximation results and concentration inequalities given in \eqref{GineA2} and \eqref{Error1}. We consider also the same kind of estimates for general Gram matrices. In section 3, we restrict ourselves to the  random finite Fourier matrix $A$ for which the law of columns and rows is uniform  We recall some precise decay rate and estimates of the eigenvalues of the Sinc kernel operator $\mathcal Q_m.$ Then, based on these last results as well as the results of section 2, we give estimates of the number of degrees of freedom and the capacity of  the associated  random matrices  $A$ and $H_m.$  In Section 4, we give some numerical examples  that illustrate the different results of this work. \\

\smallskip

We will write $\bP,$ $\bP_Z$ (resp. $\bP_Y$) depending whether we take the probability on the whole probability space, or only in $Z_1, \cdots, Z_n$ (resp. $Y_1, \cdots, Y_n$).

\medskip

\section{Approximation of a spectrum of random matrix by the spectrum of an integral operator.}

In this section, we compare   the spectra  of the matrices   $A^*A,$ $H_{\kappa}$ and the  spectrum of $T_{\kappa}$.  We recall that the entries of  $A$ are  given by ${\displaystyle A_{j,k}=\frac {\sqrt m}{n}\exp(2i\pi m  Y_jZ_k),}$ where  $Y_j, Z_k,\, j,k=1,\ldots,n$ are $2n$ independent random variables. The $Y_j$'s follow the law $P$, while the $Z_k$'s follow the law $Q$. In this case, the
 $(j, k)$ entry of $A^*A$ is equal to 
   \begin{equation} \label{A*A}
  (A^*A)_{j, k}=\frac{m}{n^2}\sum_{\ell=1}^n  \exp(2i\pi m Z_{\ell}( Y_j-Y_k), 
   \end{equation}
 so that $A^*A$ is the sum of $n$ i.i.d. random matrices. Moreover, if $H_{\kappa}= \E_Z(A^*A)$, then it is clear that 
 \begin{equation} \label{Hm} 
 H_{\kappa}=\left[H_{j,k}\right]_{1\leq j,k\leq n},\quad H_{j, k}=\E_Z(A^*A)_{j, k} =\frac 1 n\kappa(Y_j-Y_k),  \qquad  \kappa (x)= m \int e^{2i\pi m x} dQ(x).
 \end{equation} 
Let $T_{\kappa}$ be the associated kernel  integral operator, given by \eqref{kappam1}. We also recall that  $T_{\kappa}=\mathcal F^*_P\mathcal F_P,$ where $\mathcal F^*_P$ is the Fourier transform operator, given by \eqref{FiniteFourier}. Note that the roles of $P$ and $Q$ can be exchanged, without changing the spectrum. This is due to the fact that the operator $S=\mathcal F_P\mathcal F_P^*$ has the same spectrum as $T_{\kappa}.$    We start by a comparison of the sums of eigenvalues ${\displaystyle \sum_{j< d}\lambda_j(A^*A),\, \sum_{j< d}\lambda_j(H_{\kappa})}$ and  ${\displaystyle \sum_{j< d}\lambda_j(T_{\kappa})}$. This sum may be interpreted as a trace and is linked to the reconstruction error, which we explain in the following paragraph.

\subsection{Reconstruction errors.}

Let us recall the definition of the reconstruction error, when approximating an $n\times n$ matrix by its projections $P_VMP_V$ on subspaces $V$ of dimension $d\leq n.$ This error  is defined as
  $$\mathcal R_d(M)=\min\|M-P_VMP_V\|_{\scriptscriptstyle \rm HS}^2,$$
  where the minimum is taken over all subspaces $V$ of dimension $d.$ This notion is central in \cite{SCK1}.  It is well-known that the minimum is obtained when $V$ is generated by the first $d$ eigenvectors of $M^*M,$ so that
  $$\mathcal R_d(M)=\sum_{j\geq d}\lambda_j(M^*M).$$
  Equivalently,
  $$\sum_{j< d}\lambda_j(M^*M)= \max \sum_{i=0}^{d-1}\langle M^*M v_i, v_i\rangle,$$
 where the supremum is taken over all orthonormal systems of $d$ vectors. 
 We are interested in  $\mathcal R_d(A).$ Let $v_i^*$ be the system for which the maximum is obtained when $A^*A$ is replaced by $H_\kappa.$ Clearly
 $$\max \sum_{j<d} \langle A^*A v_i, v_i\rangle \geq \max \sum_{j<d} \langle A^*A v_i^*, v_i^*\rangle.$$ By applying the expectation $\E_Z$ on both sides of the previous inequality, one gets
 \begin{equation}\label{comparaison1}
 \E_Z\Big(\sum_{j<d} \lambda_j(A^*A)\Big)\geq \sum_{j<d} \lambda_j(H_{\kappa}).
 \end{equation}
In  \cite{SCK1}, by using  the covariance operator, which has the same eigenvalues as $H_\kappa$, the authors have proved that 
 \begin{equation}\label{comparaison2}
 \E\Big(\sum_{j<d} \lambda_j ( H_{\kappa})\Big)\geq \sum_{j<d} \lambda_j( T_{\kappa}).
 \end{equation}
This is valid for all integer $d>0.$  Recall that we defined  $\lambda_j(A^*A),\, \lambda_j(H_{\kappa})$ for $j\geq n$ by giving them the value $0$.  Remark that Inequality \eqref{comparaison2} is an immediate consequence of the equality of traces for $d\geq n-1.$  Indeed, 
 \begin{equation}\label{comp-traces} \sum_{k=0}^{n-1} \lambda_k(H_\kappa)= \trace(H_\kappa) =m= \trace(T_\kappa)=\sum_{k=0}^\infty\lambda_k(T_{\kappa}),
 \end{equation}
  which implies \eqref{comparaison2} for $d\geq n-1$.
  
By combining \eqref{comparaison1},\eqref{comparaison2}) and \eqref{comp-traces}, one concludes that for all integer $d>0$, we have
  \begin{equation} \label{ineq}
 \E\Big(\sum_{j< d} \lambda_j (A^*A)\Big)\geq \sum_{j< d} \lambda_j(T_{\kappa}),\qquad \E\Big(\sum_{j\geq  d} \lambda_j (A^*A)\Big)\leq \sum_{j\geq  d} \lambda_j(T_{\kappa}).
 \end{equation}
 
 \subsection{Inequalities in the quadratic means.}
 
 In this paragraph, we give  bounds for the expectations of $\|\lambda(A^*A)-\lambda(T_{\kappa})\|_{\ell^2}^2$ and  
 $\|\lambda(H_{\kappa})-\lambda(T_{\kappa})\|_{\ell^2}^2.$ For this purpose, we first need to compare the Hilbert-Schmidt norms
 $\| A^*A\|^2_{\scriptscriptstyle \rm HS},\, \|H_{\kappa}\|^2_{\scriptscriptstyle \rm HS}$ and $\|T_{\kappa}\|^2_{\scriptscriptstyle \rm HS}.$ This is partly given by the following lemma.
 
 \begin{lem} \label{normHS_A} Let   $Y_1, \cdots, Y_n$ and $ Z_1, Z_2, \cdots, Z_n$ be i.i.d with laws $P$ and $Q$, respectively. Then, we have
  \begin{equation}\label{maj-2}
  0\leq \E(\|A^*A\|^2_{\scriptscriptstyle \rm HS})- \|T_{\kappa}\|_{\scriptscriptstyle \rm HS}^2\leq \frac{2 m^2}{n}
  \end{equation}
 and
  \begin{equation}\label{maj-3}
  0\leq \E(\|H_{\kappa}\|^2_{\scriptscriptstyle \rm HS})- \|T_{\kappa}\|_{\scriptscriptstyle \rm HS}^2\leq \frac{m^2}{n}.
  \end{equation}
 \end{lem}
 
 \begin{proof}
We first prove the estimate \eqref{maj-2}.  We treat separately the diagonal terms. For the other ones, we use the expression of $A^*A$ as a sum of $n^2$ terms and distinguish between the terms for which the indices of $Z_j$ are equal (resp. different). Taking the expectation in the $Z_j$ variables and using the fact that they are independent, we find that  
\begin{eqnarray}\label{Ez}
 \E_Z(\|A^*A\|^2_{\scriptscriptstyle \rm HS})&= &\frac{m^2(2-1/n)}{n}+ \frac{m^2 n(n-1)}{n^4}\sum_{j,k=1,\,j\neq k}^n|\E_Z \left(\exp(2i\pi m Z_1(Y_k-Y_j))\right)|^2 \nonumber \\
 &= &\frac{m^2(2-1/n)}{n}+ \frac{ n(n-1)}{n^4}\sum_{j,k=1,\,j\neq k}^n|\kappa(Y_k-Y_j))|^2.
\end{eqnarray}
We take the expectation in $Y$ and get 
$$  \E_Y |\kappa(Y_k-Y_j))|^2 = \int \int |\kappa(x-y))|^2 \, dP(x) \, dP(y)= \|T_{\kappa}\|^2_{\scriptscriptstyle \rm HS}.$$
By combining the previous two equalities, one gets
 $$\E(\|A^*A\|^2_{\scriptscriptstyle \rm HS})= \frac{m^2(2-1/n)}{n}+ \left(\frac{n-1}{n}\right)^2\| T_{\kappa}\|_{\scriptscriptstyle \rm HS}^2.$$
On the other hand, from \eqref{Hm}, we have $\|T_{\kappa}\|^2_{\scriptscriptstyle \rm HS}\leq m^2.$  Collecting everything together, one gets
\eqref{maj-2}. The proof of the inequality \eqref{maj-3} is done in a similar manner. It suffices to repeat the previous steps,  use the fact that $\kappa(0)=m$ and write
$$ \E(\|H_{\kappa}\|^2_{\scriptscriptstyle \rm HS})= n \frac{m^2}{n^2}+ \frac{1}{n^2}\sum_{j,k=1,\,j\neq k}^n \E \left(|\kappa(Y_k-Y_j))|^2\right)=\frac{m^2}{n}+\frac{n(n-1)}{n^2}\|T_{\kappa}\|^2_{\scriptscriptstyle \rm HS}.$$ 
\end{proof}
Remark that we have also
 \begin{equation}\label{maj-4}
  0\leq \E_Z(\|A^*A\|^2_{\scriptscriptstyle \rm HS})- \|H_{\kappa}\|_{\scriptscriptstyle \rm HS}^2\leq \frac{ m^2}{n}
  \end{equation}

\begin{thm}\label{gineA} Under the above notation, we have  
the inequalities
 \begin{equation}\label{gineA2}
  \E\left( \|\lambda(A^*A)-\lambda(T_{\kappa})\|_{\ell^2}^2\right)\leq \frac{2m^2}{n}
  \end{equation}
  and
\begin{equation}\label{gineA2-2}
  \E\left( \|\lambda(H_{\kappa})-\lambda(T_{\kappa})\|_{\ell^2}^2\right)\leq \frac{m^2}{n}.
  \end{equation}  
  \end{thm}
    \begin{proof}
    Recall that we define  $\lambda_j(A^*A),\, \lambda_j(H_{\kappa})$ for $j\geq n$ by giving them the value $0$ for such $j$'s. We want to prove that
   \begin{equation}\label{gineA2bis}
  \E\left( \sum_{j=0}^{\infty}|\lambda_j(A^*A)-\lambda_j( T_{\kappa})|^2\right)\leq \frac{2m^2}{n},\quad \E\left( \sum_{j=0}^{\infty}|\lambda_j(H_{\kappa})-\lambda_j( T_{\kappa})|^2\right)\leq \frac{m^2}{n}.
  \end{equation} 
    The left-hand side of the previous first inequality  may be written as
  \begin{equation}\label{Ineq2}
  \sum_{j=0}^{n-1}\E(\lambda_j^2(A^*A))-\sum_{j=0}^{\infty}\lambda_j^2(T_{\kappa}) -2 \sum_{j=0}^{\infty}\lambda_j(T_{\kappa})\big(\E(\lambda_j(A^*A)-\lambda_j(T_{\kappa})\big). 
  \end{equation}
   We use  \eqref{maj-2} for the first term, and find that the sum of the two first terms is bounded by the right-hand side of \eqref{gineA2}. So it is sufficient to prove that
 \begin{equation}\label{Ineq3}
\sum_{j=0}^{\infty}\lambda_j(T_{\kappa})\big(\E(\lambda_j(A^*A))-\lambda_j(T_{\kappa})\big)\geq 0.  
   \end{equation}  
 We use an Abel transformation for partial sums of this infinite sum and obtain
 \begin{equation}
 \label{Abel}
\sum_{j=0}^{N}\big((\lambda_j(T_{\kappa})-\lambda_{j+1}(T_{\kappa})\big)\sum_{k=0}^j \Big(\E(\lambda_k(A^*A))-\lambda_k(T_{\kappa})\Big)+\lambda_{N}(T_{\kappa})\sum_{k=0}^{N} \Big(\E(\lambda_k(A^*A))-\lambda_k(T_{\kappa})\Big). 
 \end{equation}
 All terms are non negative in this sum: indeed the factors $\lambda_j(T_{\kappa})-\lambda_{j+1}(T_{\kappa})$ are non negative since the sequence of eigenvalues is non increasing and we use \eqref{ineq} for the other factors. So all partial sums are non-negative, which proves
    \eqref{gineA2}.
    
In the same way, we get \eqref{gineA2-2}. It suffices to substitute $A^*A$ by $H_{\kappa}$ in the previous proof and use the inequalities \eqref{maj-3}.
   \end{proof} 

\subsection{Concentration inequalities. }

We will now give $l^2-$approximation errors of the spectrum of $A^*A$ and of $H_{\kappa}$ by the spectrum of the operator $T_{\kappa}.$ These errors are obtained
 with  large probability. This is given by the following theorem.
 
\begin{thm}\label{Main1} Under the previous notation, for any $\xi> 0,$ we  have  the  inequalities
\begin{equation}\label{error1}
\|\lambda(A^*A)-\lambda(T_{\kappa})\|_{\ell^2} \leq \frac{\sqrt 2 m(\xi+1)}{\sqrt n}
\end{equation}

\begin{equation}\label{error2}
\|\lambda(H_{\kappa})-\lambda(T_{\kappa})\|_{\ell^2} \leq \frac{m(\xi+1)}{\sqrt n},
\end{equation}
each of them with probability $1-e^{-\xi^2}.$
\end{thm}

\begin{proof} We first prove the estimate \eqref{error1}. Given two n-tuples $\pmb y=(y_1,\ldots,y_n)$ and $\pmb z=(z_1,\ldots,z_n)$, we will use the notation  
$(\pmb y, \pmb z)=(y_1,\ldots,y_n,z_1,\ldots,z_n)$ and consider $n\times n$ matrices with entries $\exp (2i\pi m z_jy_k)$. For the sake of simplicity,  we use the same notation $A$ for such a matrix, but denote by   $\lambda_i(\pmb y,\pmb z)$ the eigenvalues $\lambda_i(A^*A), \,\,\, 0\leq i\leq n-1.$
We will use McDiarmid's concentration inequality for the $2n-$variate mapping 
$$(\pmb y, \pmb z)\mapsto \Phi(\pmb y, \pmb z)=\| \lambda(\pmb y,\pmb z)-\lambda (T_{\kappa})\|_{\ell^2}.$$
 We will prove that the previous  mapping $\Phi$ satisfies the boundedness assumption.  That is, when only one of the $2n$ coordinates differs between $(\pmb y, \pmb z)$ and $(\pmb y', \pmb z')$, the difference of values of $\Phi$ in these two sets of variables differs by at most  $\sqrt{\frac{2m^2}n}$. 
 In fact, from the triangular inequality we know that 
 $$
 \left|\Phi(\pmb y, \pmb z)-\Phi(\pmb y', \pmb z') \right|= \Big|\| \lambda(\pmb y,\pmb z)-\lambda (T_{\kappa})\|_{\ell^2}-\| \lambda(\pmb y',\pmb z')-\lambda (T_{\kappa})\|_{\ell^2}\Big|\leq \|\lambda(\pmb y,\pmb z)-\lambda(\pmb y',\pmb z')\|_{\ell^2},
 $$
 and it suffices to prove that 
 \begin{equation}
 \label{boundeddiff1}
  \|\lambda(\pmb y,\pmb z)-\lambda(\pmb y',\pmb z')\|^2_{\ell^2} \leq \frac{2m^2}{n^2}.
 \end{equation}
 Let us take for granted \eqref{boundeddiff1}, which we will check at the end of this proof. It follows from McDiarmid's inequality that
$$\bP(\Phi(Y_1, \cdots, Y_n, Z_1, \cdots, Z_n)- \E\Phi(Y_1, \cdots, Y_n, Z_1, \cdots, Z_n)>\xi)\leq e^{-\frac{2 \xi^2}{2n 2 m^2/n^2}}=e^{-\frac{n \xi^2}{2 m^2}},$$ or equivalently, that
$$\Phi(Y_1, \cdots, Y_n, Z_1, \cdots, Z_n)\leq \E\Phi(Y_1, \cdots, Y_n, Z_1, \cdots, Z_n)+\frac{\sqrt 2 m \xi}{\sqrt n}$$
with probability at least $1-e^{-\xi^2}$. To conclude for \eqref{error1}, it suffices to bound the first term of the right-hand side by using Schwarz Inequality followed by \eqref{gineA2}.

\medskip
Let us now prove the boundedness of the differences by proving \eqref{boundeddiff1}. It is sufficient to prove it when the varying variable is one of the $y_j'$s. Indeed, we know that the spectrum of $A^*A$ is the same as the spectrum of $AA^*$, so that we can exchange the role of rows and columns. So without loss of generality, we assume that $\pmb z=\pmb z'$, and that only the coordinate $y_j$ differs between $\pmb y$ and $\pmb y'$. The two matrices $A^*A$ coincide except for their $j$-th rows and columns. Let us call   $\hat \lambda(\pmb y, \pmb z)$  the ordered sequence of eigenvalues of the  matrix $\hat B,$  obtained 
 by substituting the coefficients of $j$-th row and the $j$-th column of $A^*A$ with zeros.  Note that $\hat \lambda_{n-1}(\pmb y,\pmb z)=0$ and
 $$\hat \lambda(\pmb y, \pmb z)=\hat \lambda(\pmb y', \pmb z).$$ Moreover, from the Cauchy eigenvalues interlacing property, we have
 \begin{equation} \label{interlacing}
\lambda_i(\pmb y, \pmb z) \geq \hat \lambda_i(\pmb y, \pmb z) \geq \lambda_{i+1}(\pmb y,\pmb z),\quad \forall\, 0\leq i\leq n-2.
 \end{equation}
Hence, by using the previous inequality, together with the expression of the trace of a square matrix, one gets
$$\|\lambda(\pmb y,\pmb z)- \hat\lambda(\pmb y,\pmb z)\|_{\ell^1}=\sum_i (\lambda_i(\pmb y,\pmb z)-\hat \lambda_i(\pmb y, \pmb z))= \trace(A^*A)-\trace(\hat B)=(A^*A)_{j,j}=\frac{m}{n}.$$
Since the $\ell^2$-norm is bounded by the $\ell^1$-norm, one gets
\begin{equation}\label{norme2}
\|\lambda (\pmb y,\pmb z)- \hat \lambda(\pmb y,\pmb z)\|_{\ell^2}^2 \leq    \frac {m^2}{n^2}.
\end{equation}
If we replace $\pmb y$ by $\pmb y'$, we are led to the same matix $\hat B$, so that the same inequality is valid for $\pmb y'$ in place of $\pmb y.$ Moreover, since the interlacing property \eqref{interlacing} is  valid for 
$\lambda_i(\pmb y,\pmb z)$ and $\lambda_i(\pmb y',\pmb z)$,  one gets 
\begin{eqnarray}
\label{norme2-2}
\hspace*{-2cm}|\lambda_i(\pmb y, \pmb z)-\lambda_i(\pmb y', \pmb z)|^2&\leq& \max\Big(\big(\lambda_i(\pmb y, \pmb z)-\hat\lambda_i(\pmb y, \pmb z)\big)^2, \big(\lambda_i(\pmb y', \pmb z)-\hat\lambda_i(\pmb y, \pmb z)\big)^2\Big) \nonumber\\
&\leq & \big(\lambda_i(\pmb y, \pmb z)-\hat\lambda_i(\pmb y, \pmb z)\big)^2+\big(\lambda_i(\pmb y', \pmb z)-\hat\lambda_i(\pmb y, \pmb z)\big)^2.
\end{eqnarray}
Hence, by combining \eqref{norme2} and \eqref{norme2-2}, one gets the desired inequality \eqref{boundeddiff1}.\\

\noindent
Finally, to prove the inequality  \eqref{error2}, it suffices to repeat the previous proof with $H_{\kappa}$ in place of $A^*A$  and the use of the new $n-$variate mapping $\pmb y\mapsto\| \lambda(\pmb y)-\lambda (T_{\kappa})\|_{\ell^2},\,\lambda(\pmb y)=\lambda(H_{\kappa}),$ together with the estimate 
\eqref{gineA2-2}. We leave the details for the reader.
\end{proof}

At this point, we remark that the same method allows us to prove the same kind of approximation between the spectrum of $A^*A$ and of  $H_{\kappa}$ (the expectation and the probability concern only the $n$ variables $Z_j$'s). We gather in the next theorem the quadratic means and the concentration inequalities. For quadratic means, the proof uses \eqref{comparaison1} and it is a direct adaptation of the proof of Theorem~\ref{Main1}. For concentration inequalities, we use the fact that $A^*A$ and $AA^*$ have the same spectrum, which exchanges the role of $\pmb y$ and $\pmb z.$ We then use the interlacing property of eigenvalues  to prove that the mapping $\pmb z\mapsto \|\lambda(\pmb z)-\lambda(H_\kappa)\|_{\ell^2}$ has bounded differences for fixed $\pmb y.$ 

\begin{thm}\label{A-H} Under the previous notation, we have 
\begin{equation}
\E_Z\left( \|\lambda(A^*A)-\lambda( H_{\kappa})\|^2_{\ell^2}\right)\leq \frac{m^2}{n}.
\end{equation}
Moreover,  for any $\xi > 0,$ we have
\begin{equation}
\|\lambda(A^*A)-\lambda( H_{\kappa})\|^2_{\ell^2} \leq \frac{ m(\xi+1)}{\sqrt n},
\end{equation}
with probability in $Z$ at least $1-e^{-\xi^2}$.
\end{thm}

Next, we extend our previous  results  to the  general case of  Hermitian, continuous and positive semi-definite kernel $\kappa(x,y).$ Here, the variables  $x,y $ belong
to some input space $\mathcal X,$ with $\mathcal X$ a locally compact metric space.  Moreover, we assume that  ${\kappa}(\cdot,\cdot)\in L^2_{P\otimes P}(\mathcal X  \times \mathcal X ).$ Consequently, the associated kernel integral  operator, which we still denote by $T_{\kappa}$, is a semi-positive self-adjoint Hilbert-Schmidt operator. We  still denote by $H_\kappa$ the Gram matrix 
$$ H_{\kappa}=\left[H_{j,k}\right]_{1\leq j,k\leq n},\quad H_{j, k}=\frac 1 n\kappa(Y_j,Y_k).$$
Also, note that from \cite{SCK1}, the inequality \eqref{comparaison2} is still valid for the more general 
kernel ${\kappa}(\cdot,\cdot).$ That is, 
\begin{equation}\label{comparaison3}
 \E\Big(\sum_{j<d} \lambda_j ( H_{{\kappa}})\Big)\geq \sum_{j<d} \lambda_j( T_{{\kappa}}).
 \end{equation}
 Moreover, as we have done in the proof of Lemma~1, we have 
 \begin{eqnarray*} \E(\|H_{{\kappa}}\|^2_{\scriptscriptstyle \rm HS})&= &\frac{1}{n^2}\sum_{j=1}^n \E\big(|{\kappa}(Y_j,Y_j)|^2\big)+ \frac{1}{n^2}\sum_{j,k=1,\,j\neq k}^n \E \left(|{\scriptstyle K}(Y_k,Y_j))|^2\right)\\
 &=&\frac 1n\int |\kappa(x,x)|^2 dP(x)+\Big(1-\frac 1n\Big)\iint \kappa(x,y) dP(x)dP(y).\end{eqnarray*}
By proceeding as in the proof of Theorem 1, we deduce that
\begin{equation}\label{Ineq4}
 \E\left(\|\lambda(H_{{\kappa}})-\lambda( T_{{\kappa}})\|_{\ell^2}^2\right)\leq \frac{1}{n}\left(\int |\kappa(x,x)|^2 dP(x)-\iint \kappa(x,y) dP(x)dP(y)\right).
 \end{equation}

In the next theorem, we  moreover assume that ${\displaystyle \sup_{y\in \mathcal X} {\scriptstyle K}(y,y)=R<\infty.}$ 
With this condition, we can also repeat the proof of Theorem 2 and finally obtain the 
following result, which generalizes those obtained  for  stationary kernels.

\begin{thm} \label{non-stat} Under the previous notations and  assumptions on the positive semi-definite kernel ${\kappa}(\cdot,\cdot),$ we have 
\begin{equation}\label{Ineq5}
 \E\left(\|\lambda(H_{{\kappa}})-\lambda( T_{{\kappa}})\|_{\ell^2}^2\right)\leq \frac{R^2}{n}
 \end{equation}
 and for any $\xi > 0,$ we have
\begin{equation}\label{Ineq6}
 \|\lambda(H_{{\scriptscriptstyle K}})-\lambda( T_{{\scriptscriptstyle K}})\|_{\ell^2}\leq \frac{R(\xi+1)}{\sqrt n},
\end{equation}
with probability at least $1-e^{-\xi^2}.$  Here, $H_{\kappa}$ is the Gram matrix and $T_{{\kappa}}$
 the integral operator with kernel $\kappa(\cdot,\cdot)$.
\end{thm}

The remaining  of this section is devoted to comparison of this last theorem with similar results from the literature. Inequality \eqref{Ineq4} improves Lemma 4.1 of \cite{KG} and our proof seems new and fairly simple. Nevertheless, it is based on combining Inequality \eqref{comparaison3}, given in \cite{BBZ} and  the replacement of the Gram matrix by the covariance operator, which has the same spectrum. 

It was also natural to ask for a vector valued version of McDiarmid'Inequality, which has been done by Pinelis but with a loss in the constant. It has later been used in  \cite{Rosasco1}, see also \cite{Rosasco2}, where the authors obtain that 
$$ \left \| \lambda(H_{\kappa})-\lambda(T_{\kappa})\right\|_{\ell^2}\leq \frac{2 \sqrt{2} \xi R}{\sqrt{n}},$$
with probability at least $1-2 e^{-\xi^2}.$ 
It is a little weaker than our estimate. One could also use Lemma 1 in \cite{BBZ2}, which deals with covariance operators. We tried to have the best constants and have also a small gain compared to the use of this lemma. Of course we do not know whether our results are optimal, even in the particular case that we study below, but numerical results tend to prove that they are not so far from optimality. For the use of McDiarmid's Inequality, we were inspired by  the work of  Shawe-Taylor, Cristianini and their co-authors, in which they  prove the difference boundedness for each eigenvalue by using the eigenvalues interlacing property. From our $\ell^2$ result, we can deduce that with probability at least $1-e^{-\xi^2},$  we have simultaneously
$|\lambda_j(H_\kappa)-\lambda_j(T_\kappa)|\leq 
\frac{R(\xi+1)}{\sqrt n},$
 In fact, we can do better. Indeed, using the $\ell^\infty-$norm instead of the $\ell^2-$norm for the concentration inequality, we can gain  the factor $2$ in \eqref{boundeddiff1}, since in \eqref{norme2-2} there is no need to bound the maximum by a sum.  The same is valid for $A^*A$. We state the two corresponding statements. The first one generalizes the known concentration inequalities for a single eigenvalue.
\begin{thm}
With the same assumptions as in Theorem \ref{non-stat}, for any $\xi>0,$ we have with probability at least $1-e^{-\xi^2},$ 
$$\|\lambda(H_\kappa)-\lambda(T_\kappa)\|_{\ell^\infty}\leq 
\frac{R(\frac{\xi}{\sqrt 2}+1)}{\sqrt n}.$$

\end{thm}
\begin{thm}
Under the previous notations, for any $\xi> 0,$ we  have  the inequalities
\begin{equation}\label{error-ind1}
\|\lambda(A^*A)-\lambda(T_{\kappa})\|_{\ell^\infty} \leq \frac{ m(\xi+\sqrt 2)}{\sqrt n}, 
\end{equation}
and
\begin{equation}\label{error-ind2}
\quad \|\lambda (A^*A)-\lambda(H_{\kappa})\|_{\ell^\infty} \leq \frac{m(\frac{\xi}{\sqrt 2}+1)}{\sqrt n}, 
\end{equation}
where each inequality holds  with probability at least $1-e^{-\xi^2}.$ 
\end{thm}

 \section{Degrees of Freedom and Capacity}

 In the sequel, we restrict ourselves to the case where both  sequences of  random variables $Y_j,$ $Z_k$ follow the uniform law on $I=(-1/2,1/2).$ We will be interested in the behaviour of the spectra when $m$ increases and we will take slighlty different notations for the kernel, the matrix $H_\kappa$ and the integral operator. More precisely, the Sinc-kernel is given by 
 $$\kappa_m(x)= \frac{\sin \pi mx}{\pi x}.$$
 We will use the notation $H_m$ instead of $H_{\kappa_m},$ for  the corresponding kernel matrix and $\mathcal Q_m$ the integral operator, given by
 $$\mathcal Q_m f(x)=\int_{-1/2}^{+1/2} \frac{\sin \pi m(x-y)}{\pi (x-y)}f(y) dy.$$
 We keep the notation $A$ for the random Fourier matrix.

 We first recall what is known on the spectrum of $\mathcal Q_m$. We also describe what the previous concentration  inequalities imply for the spectra of random Fourier and random Sinc-kernel matrices $A^*A$ and $H_m.$  Later on, we give estimates of the number of degrees of freedom and the capacity associated with these two last random matrices.  

 \subsection{Decay of the spectra}

 It is well known that all the  eigenvalues  of $\mathcal{Q}_m$ are smaller than $1.$ Roughly speaking, they are very close to $1$ for $j\leq m-c\log m$ and very close to $0$ for $j> m+c \log m,$ for some constant $c.$
 The in-between region  is called the plunge region.
 The asymptotic behaviour of the spectrum for $m$ tending to $\infty,$ is well-known. An asymptotic formula for its distribution has been  given by Landau and Widom (see \cite{HL}). More precisely, for $\alpha>0,$ let
 $$N_{\mathcal{Q}_m}(\alpha)=\#\{\lambda_j(\mathcal Q_m);\, \lambda_j(\mathcal Q_m)>\alpha\}.$$ Then, we have
 \begin{equation}
 \label{counteigen1}
 N_{\mathcal{Q}_m}(\alpha)=m+\left[\frac 1{\pi^2}\log\left(\frac{1-\alpha}{\alpha}\right)\right]
  \log(m)+o(\log(m)),
 \end{equation}
 and $N_{\mathcal{Q}_m}(\alpha)=0$ for $\alpha\geq 1.$
 We define similarly   $N_{H_m}(\alpha)$ and $N_{A^*A}(\alpha).$
Let us define
 \begin{equation}\label{gamma}
 \gamma_\xi= \frac {\xi}{\sqrt 2}+1,\quad \xi >0.
 \end{equation}
 From \eqref{error-ind2} of  Theorem~6,  we know that with a probability larger than $1-e^{-\xi^2},$ we have the inequality 
 $$|\lambda_j(H_m)-\lambda_j(\mathcal Q_m)|\leq \frac{\gamma_\xi m}{\sqrt n},\qquad j\geq 0.$$ For $\alpha$ fixed, we use the elementary inequality
 $$N_{\mathcal{Q}_m}\left(\alpha +\frac{\gamma_\xi m}{\sqrt n}\right)\leq N_H(\alpha)\leq N_{\mathcal{Q}_m}\left(\alpha -\frac{\gamma_\xi m}{\sqrt n}\right),$$
 that holds with probability larger than $1-e^{-\xi^2}.$ It follows that
 \begin{equation}
 \label{counteigen2}
 N_{H_m}(\alpha)=m+\left[\frac 1{\pi^2}\log\left(\frac{1-\alpha}{\alpha}\right)\right]
 \log(m)+\varepsilon(\log(m))
 \end{equation}
 with $\varepsilon$ arbitrarily small when $m$ and $\frac{n}{m^2}$ tend to $\infty.$ 
 
 The matrix $A^*A$ satisfies the same kind of estimates. This means that Landau--Widom Formula is also valid for $H_m$ and $A^*A,$ with high probability as $m^2/n$ tends to $0.$

 \smallskip

 We will also make use of  Landau's double inequality \cite{Landau2},
 \begin{equation}\label{landau}
 \lambda_{\lceil m\rceil}(\mathcal Q_m)\leq 1/2\leq \lambda_{[m]-1}(\mathcal Q_m).
 \end{equation}
 Here, $[x]$ and  $\left\lceil x\right\rceil$ refer to the integer part and to the least integer greater or equal to $x,$
 respectively. The previous inequalities say that,
 roughly speaking, $\lambda_j(\mathcal Q_m)$ goes through the value $1/2$ at $j$ too close to $m.$ An approximate result is valid for $H_m$ and $A^*A$ when $\frac{m^2}{n}$ is small.

 \smallskip

If we restrict ourselves to the main term in the  Landau-Widom Theorem, we find the approximate asymptotic decay of the eigenvalues in the plunge region after $m,$ that is
 $$\lambda_k(\mathcal Q_m)\approx \exp\left(-\frac{\pi^2(k-m)}{\log m}\right).$$
 This is heuristic and the next result, is to the best of our knowledge, the first non asymptotic exponential decay estimate of the $\lambda_n(\mathcal Q_m)$ in the right side of the plunge region. We  also give a super-exponential decay result that is proved in the same paper \cite{BJK}. 
    \begin{thm} \label{decay-th} There exist  uniform positive constants $\eta$ and $C$ such that, for $m\geq 3$ and $ k\geq m,$ one has the inequality
   \begin{equation}\label{decay}
   \lambda_k(\mathcal Q_m)\leq C \exp\left[-\eta\left(\frac{k-m}{\log m}\right)\right].
   \end{equation}
 Moreover, for any integer $k\geq \frac{e \pi}{4} m,$ we have 
 \begin{equation}\label{decay2}
   \lambda_k(\mathcal Q_m)\leq \exp\left(-(2k+1)\log\frac{4 (k+1)}{e \pi m}\right).
  \end{equation} 
   \end{thm}
   
   The second estimate is given in \cite{BJK}. The first one can be deduced from the same paper, where it is written 
   in a different way. What is proved in \cite{BJK} is such an exponential decay  for $m>22$ and $k-m>\delta \log m,$ for some uniform constant $\delta >0.$ It may be seen that the condition on $m$ can be relaxed into $m\geq 3$. The additional constant $C$ allows to make sure that it is also valid in the missing range
   $m < k < m+\delta \log m.$
   
   Remark that thanks to the rapid decay of the  $\lambda_k(\mathcal Q_m)$ given by \eqref{decay2}, for $d \geq \frac{e \pi}{4} m,$ the sums in  the right-hand side of \eqref{comparaison2} and \eqref{ineq} are bounded, up to a constant by the same quantity as $\lambda_d(\mathcal Q_m).$ Hence, by using 
   the inequalities   \eqref{comparaison2} and \eqref{ineq}, one gets the same rapid decay  for   $\E(\lambda_k(A^*A))$ and    $\E(\lambda_k(H_m))$.  For instance,  
   there exists a uniform constant $C,$ such that 
  \begin{equation}\label{decay3}
   \E(\lambda_k(A^*A)),\,\,\,\,    \E(\lambda_k(H_m))  \leq C \exp\left(-(2k+1)\log\frac{4 (k+1)}{e \pi m}\right),\quad \frac{e \pi}{4} m \leq k\leq n-1.
    \end{equation}

 \subsection{Degrees of Freedom}

In this paragraph, we are interested  in the number of degrees of freedom of a matrix or an operator, which makes sense in the area of wireless  communication networks. Different definitions have been given in this context in \cite{ST}. We give here a simple definition in terms of eigenvalues. Similar definitions may be found in approximation theory or in computational complexity theory, where it is known as a complexity number. We first show that one can easily have a good approximation  for the integral operator $\mathcal{Q}_m.$ Later on this will provide  us with a good approximation for the corresponding random matrices $H_m$ and $A^*A.$

 \begin{defn} Let $T$ be a  Hilbert-Schmidt positive semi-definite Hermitian operator. We define the (uniform) number of  degrees of freedom at level $\varepsilon$ by 
   \begin{equation}\label{deg_infty}
   \deg_\infty (T, \varepsilon)=\min \{s ; \lambda_{s}(T)\leq \varepsilon\}.
   \end{equation}
      \end{defn}
   Depending on the application in view, it makes sense to be interested  in small values of $\varepsilon,$ or values that are close to the largest eigenvalue of the integral operator, that is, close to $1$ when considering $\mathcal Q_m.$  Remark that by Landau's double inequality we have
    $$[m]-1\leq \deg_\infty (\mathcal Q_m, 1/2)\leq \lceil m\rceil.$$ For other values of $\varepsilon,$ we first need some further estimates on the Hilbert-Schmidt norm of $\mathcal Q_m.$ The coefficient of $\log(m)$ given in the next lemma is critical and error term is small. This improves the estimates given in Chapter 1 of \cite{HL},  that is, for $m>1,$
     \begin{equation}\label{Hogan}
     m- C_0\log(m) -C_1 \leq \|\mathcal Q_m\|_{\scriptscriptstyle \rm HS}^2 \leq m,
     \end{equation}
      for some constants $C_0, C_1$ independent of $m.$ 
     
  \begin{lem}\label{Estimate_norm}
   Under the previous notations and assuming that $m>1$, we have 
   \begin{equation}\label{norm_Hs}
   \|\mathcal Q_m\|_{\scriptscriptstyle \rm HS}^2= m -\frac{1}{\pi^2} \log(m) +\mathcal E,\qquad|\mathcal E|\leq  0.52.
   \end{equation}
   \end{lem}  
   
   \begin{proof} By symmetry we have
   $$
   \|\mathcal Q_m\|_{\scriptscriptstyle \rm HS}^2=\iint_{I\times I} \frac{\sin^2(\pi m(x-y))}{\pi^2(x-y)^2} \, dx \, dy=4\iint_D \frac{\sin^2(\pi m(x-y))}{\pi^2(x-y)^2} \, dx \, dy, $$
   where the domain of integration
   $D$ may be described as $ D=\{(x, y); 0<x<1/2, |y|<x\}.$ We consider  the new variables $u=x-y$ and $y$, so that  $0<u<1, -\frac u2<y<\frac 12-u$.
   Also, note that 
   $$\int_0^\infty \frac{\sin^2(\pi u)}{\pi^2u^2} du =\frac{1}{2} \int_{\mathbb R} \frac{\sin^2(\pi u)}{\pi^2u^2} du=1.$$
   The last equality is a trivial consequence of Plancherel's identity. By using the previous change of variables, together with the previous equality, one gets
  \begin{eqnarray}\label{Eq2.00}
   \|\mathcal Q_m\|_{\scriptscriptstyle \rm HS}^2 &=& 2\int_{0}^1 (1-u) \frac{\sin^2(\pi m u)}{\pi^2u^2} du  
   = \left[\int_{0}^m\frac{\sin^2(\pi u)}{\pi^2u^2} du\right]-\frac{2}{\pi^2}\left[\int_0^m\frac{\sin^2(\pi u)}{u} du\right] \nonumber\\
   &=&2m\left[  \int_0^\infty \frac{\sin^2(\pi u)}{\pi^2u^2}- \int_{u>m}\frac{\sin^2(\pi u)}{\pi^2u^2} du\right]-\frac{1}{\pi^2}\left[2\int_0^1 \frac{\sin^2(\pi u)}{u} du+\int_1^m\frac{1-\cos(2\pi u)}{u} du\right] \nonumber \\
   &=&   m-2m\int_{u>m}\frac{\sin^2(\pi u)}{\pi^2u^2} du -\frac{1}{\pi^2}\left[2\int _{0<u<1}\frac{\sin^2(\pi u)}{u} du+\log(m)-\int _{1<u<m}\frac{\cos(2\pi u)}{u} du\right].\nonumber
   \end{eqnarray} 
   We have found the two main terms and it remains to bound 
  $$|\mathcal E|\leq 2m\int_{|u|>m}\frac{\sin^2(\pi u)}{\pi^2u^2} du + \frac{1}{\pi^2}\left[2\int _{0<u<1}\frac{\sin^2(\pi u)}{u} du +\left|\int _{1<u<m}\frac{\cos(2\pi u)}{u} du\right|\right].$$ 
   The first term is bounded by $\frac{2}{\pi^2}.$ For the last one, as it is classical we perform an integration by parts and find
  $$\frac{1}{\pi^2}\left|\int _{1<u<m}\frac{\cos(2\pi u)}{u} du\right|=
 \frac{1}{\pi^2} \left|\frac{\sin(2\pi m)}{2\pi m}+\int_1^m \frac{\sin(2\pi u)}{2\pi u^2} du \right|\leq \frac{1}{\pi^3}.$$
 We compute numerically the term in the middle and find that 
  $$ \frac{2}{\pi^2} \int _{0<u<1}\frac{\sin^2(\pi u)}{u} du \approx 0.247.$$
  Collecting everything together, one gets the estimate of the error term given in \eqref{norm_Hs}. 
    \end{proof}
   By using the previous lemma, one gets the following statement, which provides us with estimates for $ \deg_\infty (\mathcal Q_m, \varepsilon).$
   \begin{lem} Let $m\geq 4$ be a positive real number. For $\varepsilon < 1/2,$ the  number of  degrees of freedom satisfies the inequalities
 \begin{equation}\label{deg_inftyQ} m-1\leq  \deg_\infty (\mathcal Q_m, \varepsilon)\leq m+\varepsilon^{-1}\log m.
 \end{equation}  For $1/2<\varepsilon <1, $ these inequalities are replaced by
 \begin{equation}\label{deg-big}
 m-(1-\varepsilon)^{-1}\log m\leq  \deg_\infty (\mathcal Q_m, \varepsilon)\leq \lceil m\rceil.
 \end{equation}
  \end{lem}
 \begin{proof}
 The left-hand side of the first inequality, as well as the right-hand side of the second one, follow from Landau's double inequality \eqref{landau}. 
 To prove the second inequality, we first know  from Lemma \ref{Estimate_norm} that, for $m\geq 4,$
 \begin{equation}\label{hogan}
 m-\frac{1}{2}\log m\leq \|\mathcal Q_m\|_{\scriptscriptstyle \rm HS}^2 =\sum_{j= 0}^{\infty}\lambda_j(\mathcal Q_m)^2.
 \end{equation}
  On the other hand, we have the identity 
 $$
 \trace(\mathcal Q_m)=\sum_{j= 0}^{\infty}\lambda_j(\mathcal Q_m)=m.
 $$
 It follows that
 \begin{equation}\label{HL0}
 \sum_{j\geq 0}\lambda_j(\mathcal Q_m)(1-\lambda_j(\mathcal Q_m))
 =\trace(\mathcal Q_m)-\|\mathcal Q_m\|_{\scriptscriptstyle \rm HS}^2
 \leq \frac{1}{2}\log m.
 \end{equation}
   Moreover, by using \eqref{HL0} and \eqref{landau}, we obtain the two inequalities
  \begin{eqnarray} \label{HL1}
  \sum_{j\leq m-1}(1-\lambda_j(\mathcal Q_m))&\leq&  \log m \\
   \sum_{j\geq m+1}\lambda_j(\mathcal Q_m)&\leq&  \log m. \label{HL3}
  \end{eqnarray}
  As a consequence of \eqref{HL1}, we have the inequality
   $$\#\{j\leq m-1\,,\lambda_j(\mathcal Q_m)<\varepsilon\}\leq \frac{ \log m}{1-\varepsilon}.$$
   The left-hand side of \eqref{deg-big} follows at once. The remaining inequality, that is, the right-hand side of \eqref{deg_inftyQ}, follows from a similar argument, with \eqref{HL3} used in place of \eqref{HL1}.
 \end{proof}
 \begin{rem} In the inequality \eqref{deg_inftyQ}, it is possible to replace $\varepsilon^{-1}$ by $C\log (\varepsilon^{-1}),$ 
 for some uniform constant $C>0.$ This follows from the decay of eigenvalues given in \eqref{decay}.
    \end{rem}
 We summarize the previous lemma by saying that, for $\varepsilon<1$, 
 \begin{equation}\label{resume}
 \deg_\infty (\mathcal Q_m, \varepsilon)=m(1+\mathcal E),\qquad{\rm with}\quad |\mathcal E|\leq \frac{\log m}{\varepsilon(1-\varepsilon)m}.
 \end{equation}
 
 Let us now consider the degrees of freedom of the random matrices $H_m$ and $A^*A,$ which we define as follows.
 \begin{defn} Let $M$ be a  positive semi-definite random Hermitian matrix. We define the numbers of degree of freedom at level $\varepsilon$ and confidence level $\alpha$ by
   \begin{equation}\label{deg_rand}
    \deg_\infty (M, \varepsilon, \alpha)=\min \{s ; \lambda_{s}(T)\leq \varepsilon \quad \mbox {\rm with probability}\, \geq \alpha\}.
   \end{equation}
   \end{defn}
One could use \eqref{error-ind1}  to find bounds for $\deg_\infty (M, \varepsilon, e^{-\xi})$. This requires that $m^2/n$ is small enough, while we only want to assume that $m/n$ is small enough. In this direction, we note that 
if  $T_1, T_2$ are   two self-adjoint Hilbert-Schmidt operators  and if $0<\varepsilon_1<\varepsilon_2,$ then we have
\begin{equation}
\label{Comp_degrees}
\deg_\infty ( T_2, \varepsilon_2)\leq \deg_\infty (T_1, \varepsilon_1)+\frac{\sum_{j\geq 0} (\lambda_j(T_1)-\lambda_j(T_2))^2}{(\varepsilon_1-\varepsilon_2)^2}.
\end{equation}
Indeed, the previous inequality follows from the fact that 
\begin{eqnarray*}
\#\{j>\deg_\infty (T_1, \varepsilon_1)\,,\lambda_j(T_2)>\varepsilon_2\}&\leq& \#\{j>\deg_\infty (T_1, \varepsilon_1)\,,\lambda_j(T_2)-\lambda_j(T_1)>\varepsilon_2-\varepsilon_1\}\\
&\leq &\frac{\sum_{j\geq 0} (\lambda_j(T_1)-\lambda_j(T_2))^2}{(\varepsilon_1-\varepsilon_2)^2}.
\end{eqnarray*}
For $0<\varepsilon<1$, we use Inequality \eqref{Comp_degrees} with $T_1=A^*A,\, T_2=\mathcal Q_m$ (resp. $T_1=H_m$) and $\varepsilon_2=\varepsilon,\, \varepsilon_1=\varepsilon/2$ to give a bound above for $\deg_\infty(A^*A, \varepsilon, e^{-\xi}).$ 
Then we exchange the role of $\mathcal Q_m$ and $A^*A$ (resp. $H_m$) and take $\varepsilon_1=\varepsilon, \,\varepsilon_2=\min(2\varepsilon, \frac{1+\varepsilon}{2})$ for the bound below. Using \eqref{resume},  one gets the following proposition. 
   \begin{prop} For $\varepsilon, \xi >0,$ we have
   $$ \deg_\infty (H_m, \varepsilon, e^{-\xi})=m(1+\mathcal E_{H_m}) \quad , \deg_\infty (A^*A, \varepsilon, \alpha)=m(1+\mathcal E_A),$$
   with $|\mathcal E_{H_m}|, |\mathcal E_A| \leq C_{\varepsilon}\left( \frac {m(\xi^2+1)}{n }+\frac{\log m}{m}\right).$ Moreover, the constant   $C_{\varepsilon}$ can be bounded by $\frac{C}{\varepsilon ^2(1-\varepsilon)^2},$ where $C>0$ is a uniform constant.
\end{prop}
    The relative errors are small for $m$ large and $n/m$ large, as we wanted to prove. Asymptotically, when these two quantities tend to $\infty,$ we have $\deg_\infty (H_m, \varepsilon, \delta)\sim m$ and $\deg_\infty (A^*A, \varepsilon, \delta)\sim m$ for fixed $\varepsilon$ and $\delta.$

 \subsection{Capacity of a random matrix}
 We proceed as we have done for the degrees of freedom by considering first the integral operator. Let us define, for  $s>0,$ the capacity
 associated with the Sinc-kernel integral operator $\mathcal Q_m,$
 $$C_{\mathcal Q_m}(s)=\sum_{k\geq 0}\log\big(1+s\lambda_k(\mathcal Q_m)\big).$$
 We claim the following, which  says that, for $m$ tending to infinity, the capacity is well approximated by the  one of an operator that has $m$ non zero eigenvalues, all  equal to $1$.
 \begin{prop} \label{capQ} There exists a positive constant $\delta$ such that, for any $m\geq 4$ and any $s>0,$ we have
 \begin{equation}
 \label{CapacityQ}
 C_{\mathcal Q_m}(s)=m \log(1+s)(1+\mathcal E),\qquad |\mathcal E|\leq \delta \frac{\log (m)}m  (1+\log_+(s)).
 \end{equation}
   \end{prop}
   \begin{rem} The constraint $m\geq 4$ could be relaxed in proofs, but for smaller values of $m$ the error terms, which contain $\log(m)/m,$ are very large. This is why we assume that $m\geq 4$ as for degrees of freedom.
   \end{rem}
 \begin{proof} Three cases are  considered in the proof. The constant $\delta$ may vary from line to line.
 \begin{enumerate}
 \item {\sl Bound above for $s\leq 2.$} Since all eigenvalues are smaller than $1$, 
 $$ C_{\mathcal Q_m}(s)\leq m \log(1+s)+ \sum_{k\geq m }\log(1+s\lambda_k(\mathcal Q_m)).$$
 The second term is bounded by $ s\sum_{k\geq m }\lambda_k (\mathcal Q_m)$, which is bounded by $s(\log(m)+1)$ because of \eqref{HL3}. Since $s\leq 2$, there exists $\delta$ such that $s\leq \delta \log(1+s),$ which allows to conclude.
  \item {\sl Bound above for $s \geq  2.$} The inequality \eqref{decay} implies that, for an integer  $\ell >m$, we have 
 $$\sum_{k\geq \ell}\lambda_k(\mathcal{Q} _m)\leq \frac{C\log(m)}{\eta}e^{-\frac{\eta(\ell-m)}{\log(m)}}.$$
 We define $m_{s}$ as 
 $$m_s =\min \{\ell \in \mathbb N,\,\, \ell\geq m  +\eta^{-1}\log(m)\log(s)\}.$$
  It follows from the previous inequality that $\sum_{k\geq m_s}\lambda_k(\mathcal{Q} _m)\leq \frac{\delta \log(m)}{s} $ for some uniform constant $\delta.$ We modify slightly the proof given for $s\leq 2,$ by writing that
\begin{equation}\label{CapacityQQ}
C_{\mathcal Q_m}(s)\leq m_s \log(1+s)+ \sum_{k\geq m_s} \log \left(1+s \lambda_k(\mathcal Q_m)\right).
\end{equation} 
The last term is bounded by $\delta \log (m).$ Consequently, we have
\begin{eqnarray*}
 C_{\mathcal Q_m}(s)- m\log(1+s)&\leq& (m_s-m)\log(1+s) +\delta \log(m)\\
 &\leq & \eta^{-1} \log(m)\log(1+s)\log(s) +\delta \log(m)\\
 &\leq & \big(\log(1+s)\big)^2 \log(m)\left(1+\frac{\delta}{(\log(1+s))^2}\right).
\end{eqnarray*} 
 This allows us to conclude for this  case.

 \item {\sl Bound below for $s >0.$} We have 
 \begin{eqnarray*}
  C_{\mathcal Q_m}(s)&\geq&  \sum_{k<[m]} \log \left(1+s \lambda_k(\mathcal Q_m)\right)= \sum_{k<[m]} 
  \log \left(\frac{1+s \lambda_k(\mathcal Q_m)}{1+s} (1+s)\right)\\
  &\geq& [m] \log(1+s)+ \sum_{k<[m]} \log \left( \frac{1+s \lambda_k(\mathcal Q_m) }{1+s}\right).
 \end{eqnarray*}
 All the terms  $\lambda_k(\mathcal Q_m)$ of this expression are bounded below by $1/2.$ Since $\log(1-x)\geq -2x$  when $ 0\leq x\leq 1/2,$ one gets
$$
 \log\left(\frac{1+s\lambda_k(\mathcal Q_m)}{1+s}\right)=\log\left(1-\frac{s}{1+s}(1-\lambda_k(\mathcal Q_m))\right)
 \geq -\frac{2s}{1+s}(1-\lambda_k(\mathcal Q_m)).$$
 But, from  \eqref{HL1}, we have 
 $$\sum_{k<[m]}(1-\lambda_k(\mathcal Q_m))\leq \log(m).$$ We conclude as before, using the fact that $\frac{2s}{1+s}\leq \delta\log(1+s)$ for all $s>0,$ so that 
 $$-\frac{2s}{1+s} \sum_{k<[m]}(1-\lambda_k(\mathcal Q_m) \geq -\delta \log(1+s) \log(m).$$
 Consequently, we have 
 \begin{equation}\label{Ineq3.4}
  C_{\mathcal Q_m}(s) \geq m \log(1+s) -\delta \log(1+s) \log(m). 
 \end{equation}

 \end{enumerate}
  \end{proof}

 We now pass from $\mathcal Q_m$ to $A^*A.$ In the sequel, we define the capacity of the matrix $A^*A$ by
 \begin{equation}
 \label{CapacityA}
 C_{A^*A}(s)= \sum_{k\geq 0} \log\big(1+s \lambda_k(A^*A)\big).
 \end{equation}
 In order to get an estimate of the previous capacity, we need the following technical lemma.
 
 \begin{lem}
 For any real $m\geq 4,$ we have the two inequalities
 \begin{equation}\label{Ineq3.1}
\E  \Big(\sum_{k<m} \left(\lambda_k(A^*A)-1\right)_+\Big) \leq \sqrt\frac{2(m+1)}{{n}},
  \end{equation}
   \begin{equation}\label{Ineq3.1bis}
\E  \Big(\sum_{k<m} \left(1-\lambda_k(A^*A)\right)_+\Big) \leq \log(m)+1+\sqrt\frac{2(m+1)}{{n}},
  \end{equation}
 Moreover, let $M_{\frac 12}$ be the (random) number of the integers  $k<m$ such that $\lambda_k(A^*A)<1/2.$ Then
 \begin{equation}\label{Ineq3.2}
 \E(M_{\frac 12}) \leq\frac{ 16 m^2}{n}+8 \log(m).
   \end{equation}
 \end{lem}
 
 \begin{proof} To prove \eqref{Ineq3.1}, we first recall that $\lambda_k(\mathcal Q_m)<1$ and only those eigenvalues  $\lambda_k(A^*A)>1$ contribute to the sum in \eqref{Ineq3.1}. By a straightforward application of Cauchy-Schwarz inequality and by using Theorem~1,  we have
 \begin{eqnarray*}
\E\,  \Big(\sum_{k<m} \left(\lambda_k(A^*A)-1\right)_+\Big)&\leq& \sqrt{m+1}\,  \E\, \left(\sum_{k<m} \left(\lambda_k(A^*A)-1\right)_+^2\right)^{1/2}\\
&\leq& \sqrt{m+1}\, \E \, \left(\sum_{k<m} \left(\lambda_k(A^*A)-\lambda_k(\mathcal Q_m)\right)^2\right)^{1/2} \leq  m\sqrt\frac{2(m+1)}{{n}}.
  \end{eqnarray*}
To prove \eqref{Ineq3.1bis}, we write 
$$\E  \Big(\sum_{k<m} \left(1-\lambda_k(A^*A)\right)_+\Big)=\sum_{k<m} (1-\lambda(\mathcal Q_m))+\E \sum_{k<m}\big(\lambda(\mathcal Q_m)-\lambda_k(A^*A)\big)+\E  \Big(\sum_{k<m} \left(\lambda_k(A^*A)-1\right)_+\Big).$$
The first term is bounded by $\log(m)$ because of \eqref{HL1} and  the second one is non positive because of 
  \eqref{ineq}. Moreover, the third term is given by \eqref{Ineq3.1}.  
  
To prove \eqref{Ineq3.2}, we start from the inequality 
$$\frac{M_{\frac 12}}{4} \leq  \sum_{k<m}\left(1-\lambda_k(A^*A))\right)^2\leq 2 \sum_{k<m}\left(1-\lambda_k(\mathcal Q_m)\right)^2+2\sum_{k<m}\left(\lambda_k(\mathcal Q_m)-\lambda_k(A^*A))\right)^2.$$
 We take the expectation and find that
   $$
   \E (M_{\frac 12}) \leq  8\sum_{k<m}\left(1-\lambda_k(\mathcal Q_m)\right)^2+8\sum_{k<m}\E\left(\lambda_k(\mathcal Q_m)-\lambda_k(A^*A))\right)^2.$$
The first sum is bounded by $\log( m)$ by \eqref{HL1}. The second one is bounded by $2 m^2/n$ by Theorem~1,
   which allows to conclude. 
 \end{proof}
 
 We can now state the main proposition of this section.
  
 \begin{prop} There exists a positive constant $\delta$ such that for all values of $s>0,$  $m\geq 4$ and $n>m,$ we have the following 
 approximation for the expectation of the capacity defined in \eqref{CapacityA},
 \begin{equation}\label{CapacityAA}
 \E( C_{A^*A}(s))= \E \Big(\sum_{k\geq 0} \log\big(1+s \lambda_k(A^*A)\big)\Big)= m\log\left(1+s\right)(1+\mathcal E),
 \end{equation}
where
 $$ |\mathcal E| \leq \delta\Big(\sqrt{\frac{m}{n}}+\frac{\log(m)}{m}\Big)\big(1+ \log_+(s)\big).$$ In particular, when $m$ and $n/m$ tend to $\infty,$
 $$\E(C_{A^*A}(s))\sim m \log\left(1+s\right).$$
 \end{prop}
 \begin{proof} 
 Let us first prove the bound above. We first assume that $s\leq  2$ and mimic the proof given for $C_{\mathcal Q_m}$ by cutting the sum   as we did in  \eqref{CapacityQQ}. As before, we can bound  terms by $\log(1+s)$ when $\lambda_k(A^*A)$ is bounded by $1.$ But  now some eigenvalues $\lambda_k(A^*A)$ may be larger than $1.$ This leads to a third term in the sum.   Namely, we write
  $$ \E C_{A^*A}(s)\leq m \log(1+s)+ \E\sum_{k<m, \lambda_k(A^*A)>1} \log\left(\frac{1+s\lambda_k(A^*A)}{1+s}\right)+\E\sum_{k\geq m }\log(1+s\lambda_k(A^*A)).$$
We  conclude for the third  term by using the fact that  the sum $\sum_{k\geq m}\E(\lambda_k(A^*A))$ is bounded by the corresponding sum  $\sum_{k\geq m}\lambda_k(\mathcal Q_m).$ This is the second inequality in \eqref{ineq}.  The second term is bounded by
$\frac{s}{1+s}\E \sum_{k<m}(\lambda_k(A^*A)-1)_+.$ Using \eqref{Ineq3.1}, one concludes that this term is bounded by $\sqrt{\frac{2(m+1)}{n}} .$ 
 Collecting everything together,
one gets a bound above for the error term $\mathcal E,$ when $s\leq 2.$  
The previous technique is also applied for  $s>2$ to yield
the required  bound above for $\mathcal E.$ The only difference is the fact that one needs the analogue of \eqref{Ineq3.1} with $m_s$ in place of $m.$ It is easily seen that this may be done with a supplementary factor ${\displaystyle \frac{m_s}{m}},$ when we consider the second term, that is a factor $\delta \sqrt{\log(1+s)}.$

For the bound below for $s>0,$ we  use the same techniques as in the corresponding proof  for the capacity $C_{\mathcal Q_m},$
but $[m]$ substituted with $m-M_{\frac 12}.$ Here,  $M_{\frac 12}$ has been defined in the previous lemma and is such that    $\lambda_k(A^*A)\geq 1/2$ for $k< m-M_{\frac 12}.$  In this case, we have 
\begin{eqnarray*}
C_{A^*A}(s)&\geq & (m- M_{\frac 12}) \log (1+s)+\sum_{k<m-M_{\frac 12}} \log \left( \frac{1+s \lambda_k(A^*A) }{1+s}\right)\\
 &\geq & (m-M_{\frac 12} )\log(1+s) +\sum_{k< m- M_{\frac 12}, \lambda_k(A^*A)<1} \log\Big(1-\frac{s}{1+s}(1-\lambda_k(A^*A))\Big)\\
 &\geq & m\log(1+s)-M_{\frac 12}\log(1+s)-\frac{2s}{s+1}\sum_{k <m} (1-\lambda_k(A^*A))_+.
\end{eqnarray*}
We take the expectations. 
To conclude for the bound below, it suffices to use the previous lemma.
 \end{proof}

 The previous  proposition deals with expectation. It is easy to deduce estimates with high probability by just using Markov Inequality. One may be tempted  to use McDiarmid's inequality for the mapping defined by the capacity, as in \cite{Ozgur}, to find accurate estimates as in Section 2.  From the computations made  in \cite{Ozgur}, it follows that, for $\xi>0$ and  with probability larger than $1-2e^{-\xi^2},$ we have 
 $$|C_{A^*A}(s)-\E(C_{A^*A}(s))|\leq 2\sqrt n\xi \log (1+sm).$$
 This is small compared to the principal term when $m$ is much larger than $\sqrt n.$ 
 
 Let us consider in particular $C(n)=C_{A^*A}(n^{2-\gamma}),$ under the assumption that $m=n^\gamma$ for $1/2<\gamma<1,$ as in the introduction. Then we get that, with probability larger than $1-2e^{-\xi^2},$ we have the estimate
 $$\left|\frac{C(n)}{(2-\gamma) n^{\gamma}\log(n)}-1 \right|\leq \beta (\xi+1) n^{\frac{\gamma-1}{2}} \log(n),$$
 for some uniform constant  $\beta.$

 \section{Numerical examples}

 We first note that due to the sub-Gaussian behaviour of the probability distribution of the  spectra approximation errors, one gets 
 relatively small errors with large probability. For example, for the concrete value of $\alpha=0.01,$ that is the error bounds of Theorem~\ref{Main1} are obtained with a probability $0.99$ as soon as $\xi\geq \xi_{0.01}\approx 2.15.$ Moreover, the errors \eqref{error1} and \eqref{error2} are smaller than $\sqrt{m}$ as soon  as $n/m \geq 20$ and $n/m \geq 10,$ respectively.\\

\noindent
  The simulations below are done with  samples of sizes $2 n$ of the uniform law on $(-1/2, + 1/2)$ in place of the $Z_j$'s and the $Y_k$'s for different values of $n$.
 \\

 \noindent
 {\bf Example 1:} In this first example, we illustrate the results of  Theorem~2 and Theorem~3. For this purpose, 
 we have considered the mean of the spectra of $(A^*A)$ and $H_{m},$ obtained from 10 realizations,  with $n=300$ and different values of
  $2\leq m\leq 20.$ Since for each value of $m,$  there is approximately $m$ significant eigenvalues, then we have computed the approximation relative $ \ell^2-$errors, given by  $\frac{1}{\sqrt{m}} \|\lambda (A^*A)-\lambda (H_m)\|_{\ell^2},$  $\frac{1}{\sqrt{m}} \| \lambda (H_m)-\lambda (\mathcal Q_m)\|_{\ell^2}$ and
    $\frac{1}{\sqrt{m}} \|\lambda (A^*A)-\lambda (\mathcal Q_m)\|_{\ell^2}.$ Also, we have computed the magnitude of the corresponding theoretical relative error, given by the quantity ${\displaystyle \sqrt{\frac{m}{n}}.}$ Recall that $\|\lambda(\mathcal Q_m)\|_{\scriptscriptstyle \rm HS}\approx \sqrt{m}.$   The obtained numerical results  are given by Table 1. These numerical results indicate that the approximation errors between the spectra    $\lambda(A^*A)$ and $\lambda(H_m)$ is smaller than the approximation errors corresponding to the spectra 
     $\lambda(A^*A)$ and $\lambda(\mathcal Q_m),$ as well as $\lambda(H_m)$ and $\lambda(\mathcal Q_m).$
   \begin{center}
    \begin{table}[h]
    \vskip 0.02cm\hspace*{2cm}
    \begin{tabular}{ccccc} \hline
     $n=300$ 
         &$\frac{\| \lambda (A^*A)- \lambda (H_m)\|_{\ell^2}}{\sqrt{m}} $&$\frac{\| \lambda (H_m)-\lambda (\mathcal Q_m)\|_{\ell^2}}{\sqrt{m}}$&$\frac{\|  \lambda (A^*A)-\lambda (\mathcal Q_m)\|_{\ell^2}}{\sqrt{m}} $&$\sqrt{\frac{m}{n}}$ \\   \hline
      $m=2$  &$02.09\%$   &$00.28 \%$   & $02.38\%$& $08.16 \%$ \\
      $m=4$  &$01.83 \%$   &$06.02 \%$   & $04.67 \%$& $11.54\%$   \\
      $m=6$  &$01.85 \%$   &$09.62 \%$   & $10.77  \%$& $14.14\%$  \\
      $m=10$ &$03.32\%$   &$13.36 \%$   & $16.34 \%$& $18.26 \%$\\
      $m=20$ &$09.42 \%$   &$22.89 \%$   & $30.57 \%$& $25.82 \%$  \\  \hline
        \end{tabular}
    \caption{Illustrations of the mean over $10$  realizations for the  $\ell^2-$errors, given by  Theorem~2 and Theorem~3, for the Sinc-kernel with  $n=300$  and different values       of $m.$ }
    \end{table}
    \end{center}
  
  \noindent
 {\bf Example 2:} In this example, we have considered the random matrices $A^*A$ and  $H_m,$ as described by section~3, with
 $n=300$ and different values of the bandwidth $2\leq m\leq 20.$  In Figure 1 (a), we have plotted the eigenvalues $(\lambda_{j}(A^*A))_{0\leq j\leq 35 }$ of the random matrix $A^*A,$ arranged in the decreasing order, versus the eigenvalues of $\mathcal Q_m.$  Then, we have repeated the previous numerical tests with the random matrix $H$ instead of the matrix $A^*A.$ The obtained numerical results are given by Figure 1(b). Note that as predicted by proposition 3, the matrices $A^*A$  and $H_m,$ each has $m$ significant eigenvalues. \\
  Also in order to check the decay of
 the eigenvalues of the random matrices $A^*A$ and  $H_m,$     we have plotted in Figures 2 (a) and 2(b),
 the graphs of  $\log(\lambda_j(A^*A))$ and  $\log(\lambda_j(H_m)).$   Note that as predicted by our theoretical results, the eigenvalues of the random matrices $A^* A$ and $H_m$  have  fast decays, starting around $k=m.$ Note that the first few eigenvalues $\lambda_j(A^*A)$ are relatively  larger  than  the corresponding eigenvalues $\lambda_j(\mathcal Q_m).$ We should mention that up to know, we do not have a satisfactory explanation to this local behaviour.

    \begin{figure}[h]\hspace*{1.5cm}
  {\includegraphics[width=13.05cm,height=4.5cm]{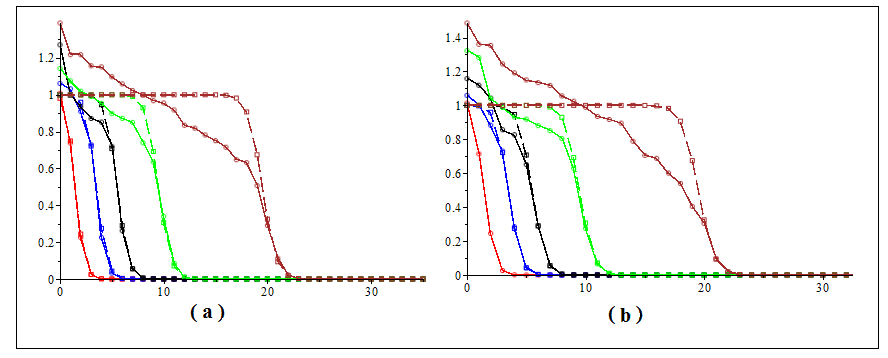}}
  \vskip -0.5cm\hspace*{1cm} \caption{ (a) Graphs of  $ \lambda(A^* A)$ (circles) versus $\lambda(\mathcal Q_m) $ (boxes) with $n=300$
   and  for the  various values of $m=2,4,6, 10, 20,$ (from the left to the right), (b) same as (a) with $ \lambda(H_m)$ instead of  $\lambda(A^* A).$ }
  \end{figure}
  \begin{figure}[h]\hspace*{1.5cm}
  {\includegraphics[width=13.05cm,height=4.5cm]{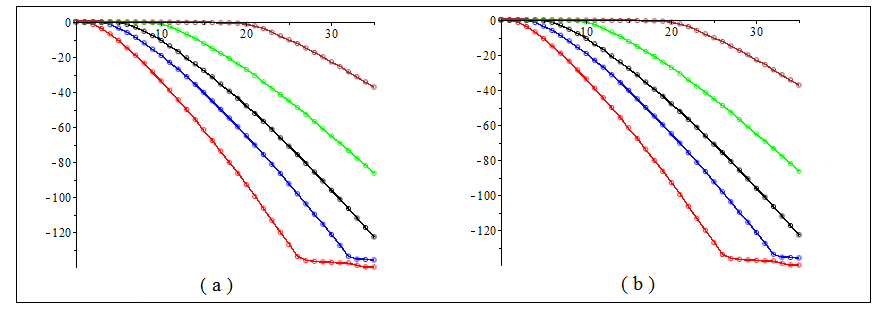}}
  \vskip -0.5cm\hspace*{1cm} \caption{  (a) Graphs of  $\log(\lambda_j(A^*A))$ with $n=300$
    and  for the  various values of $m=2,4,6, 10, 20,$ (from the left to the right), (b) same as (a) with $\log( \lambda_j(H_m))$ instead of  $\log( \lambda_j(A^*A)).$   }
  \end{figure}

    \noindent
  {\bf Example 3:} In this last example, we illustrate our theoretical estimate for the network capacity. We recall that
  this capacity is given by equation \eqref{Capacity}.
  To illustrate the previous bound estimate of the network capacity, we have considered the value of $n=300$
   and the four values of $m=2, 4, 6, 10, 20.$ Then, we  have computed the eigenvalues  $\frac{n^2}{m} \lambda_j(A^*A)$ of the  matrices $\frac{n^2}{m} A^*A.$   In Table 2, we have listed the values
  of
  \begin{equation}
  \label{capacity4}
  C(n)=\log \det\left(I_n+ \frac{n^2}{m} A^*A\right) =\sum_{j=1}^n \log\left(1+\frac{n^2}{m} \lambda_j(A^*A)\right),
  \end{equation}
   as well as the values of the corresponding estimations $\widetilde C(n)$ as well as the corresponding relative errors $E(n),$  given by
  \begin{equation}
  \label{errors}
 \widetilde C(n)=  m\log\left(\frac{n^2}{m}\right),\quad  E(n)=\frac{|C(n)-\widetilde C(n)|}{C(n)},\quad .
  \end{equation}

 \begin{center}
  \begin{table}[h]
  \vskip 0.02cm\hspace*{3cm}
  \begin{tabular}{cccc} \hline
   $m$ &$C(n)$& $\widetilde C(n)$ & $E(n)= \frac{|C(n)-\widetilde C(n)|}{C(n)}$\\   \hline
  $m=2$  &$42.05$   &$21.42 $ & $0.49$  \\
  $m=4$  &$60.59 $   &$40.08$ & $0.34$  \\
  $m=6$  &$78.94$   &$57.69 $ & $0.27$  \\
  $m=10$  &$112.18 $   &$91.05 $ & $0.19$    \\
  $m=20$ &$188.20$   &$168.24 $  & $0.11$  \\\hline
   \end{tabular}
  \caption{Illustrations of our bound estimate of the network capacity $C(n),$ given by \eqref{capacity4} and for different values of
  $m.$}
  \end{table}
  \end{center}

  \noindent
 {\bf Acknowledgements.} The authors thank Marguerite Zani and Radoslaw Adamczak for useful comments and discussions.


\begin{thebibliography}{99}

\bibitem{Adamczak} R. Adamczak, W. Bednorz, Some remarks on MCMC estimation of spectra of integral operators, {\it Bernoulli,} {\bf 21} (4) (2015), 2073--2092.


\bibitem{BBZ} G. Blanchard, O. Bousquet and L. Zwald,    Statistical properties of Kernel Prinicipal Component
Analysis, {\it Machine Learning,} {\bf 66} (2), (2007),  259--294.



\bibitem{Bonami-Karoui2} A. Bonami and A. Karoui,
Uniform approximation and explicit estimates of the Prolate Spheroidal Wave Functions, {\it Constr. Approx.} {\bf 43} (2016), 15--45. 


\bibitem{Bonami-Karoui3} A. Bonami and A. Karoui, Spectral Decay of Time and Frequency Limiting Operator, {\it Appl. Comput. Harmon. Anal.} 	{\bf 42} (2017), 1--20.

\bibitem{BJK} A. Bonami, P. Jaming and A. Karoui, Non-Asymptotic behaviour of the spectrum of the Sinc-kernel operator and related applications,  available at arXiv1804.01257, (2018).

\bibitem{DLP}
M. Desgroseilliers, O. L\'ev\^eque and  E. Preissman,  Spatial degrees of freedom of MIMO systems in line-of-sight environment, in Proc., IEEE Intl. Symposium on Information Theory, Istanbul, Turkey, Jul. 2013.

\bibitem{DLP2} M. Desgroseilliers, O. L\'ev\^eque and  E Preissman,  Partially random matrices in line-of-sight wireless networks, in Proc., IEEE Asilomar, Pacific Grove, CA, Nov. (2013).



\bibitem{Ferreira} J.C. Ferreira and  V. A. Menegatto, Eigenvalues of integral operators defined by smooth positive definite kernels,
{\it Integr. Equat. Oper. Th.,} {\bf 64} no.1,  (2009), 61--81. 

\bibitem{HL}  J. A. Hogan and J. D. Lakey, {\it Duration and Bandwidth Limiting: Prolate Functions, Sampling, and Applications,}
Applied and Numerical Harmonic Analysis Series, Birkh\"auser, Springer, New York, London, 2013.

\bibitem{Kato} T. Kato, Variation of discrete spectra, {\it Commun. Math. Phys.} {\bf 111} (1987), 501--504.

\bibitem{KG} V. Koltchinskii and E. Gin\'e, Random matrix approximation of spectra
of integral operators, {\it  Bernoulli,}  {\bf  6} (1), (2000), 113--167.

\bibitem{Landau2} H. Landau, On the density of Phase-Space Expansions, {\it IEEE Trans. Inf. Theory,} {\bf 39} no. 4, (1993), 1152--1156.

\bibitem{Landau} H. J. Landau and H. O. Pollak, Prolate spheroidal  wave functions, Fourier analysis and
uncertainty-III. The dimension of space of essentially time-and
band-limited signals, {\it Bell System Tech. J.} {\bf 41}, (1962),
1295--1336.

\bibitem{Osipov} A. Osipov, V. Rokhlin and H. Xiao, {\it  Prolate spheroidal wave functions of order zero. Mathematical tools for bandlimited approximation}, Applied Mathematical Sciences, 187, Springer, New York, 2013. 


\bibitem{Ozgur} A. \"Ozg\"ur, O. L\'ev\^eque and  D. Tse, Spatial Degrees of Freedom of Large Distributed MIMO Systems and Wireless Ad Hoc Networks, {\it IEEE J. Sel. Areas Commun.,} {\bf 31} (2) (2013), 202--2014.

\bibitem{Pinelis} I. Pinelis, An approach to inequalities for the distributions of infinite-dimensional martingales,  In: Dudley R.M., Hahn M.G., Kuelbs J. (eds) Probability in Banach Spaces, 8: Proceedings of the Eighth International Conference. Progress in Probability, {\bf  30}
 Birkh\"auser, Boston, MA, 1992, 128--134. 

\bibitem{Rosasco2} L. Rosasco, M. Belkin and E. De Vito, On Learning with Integral Operators, {\it J. Mach. Learn. Res.} {\bf 11} (2010), 905--934.

\bibitem{SCK1} J. Shawe-Taylor, K. I. Williams, N.  Cristianini and J. Kandola, On the Eigenspectrum of the Gram Matrix and the Generalized Error of Kernel-PCA, {\it IEEE Trans. Inf. Theory,} {\bf 51} no. 7, (2005), 2510--2522.


\bibitem{S-TC} J. Shawe-Taylor and N. Cristianini,  Kernel Methods for Pattern Analysis, Cambridge
University Press, Cambridge (2005).

\bibitem{S-TCK} J. Shawe-Taylor, N. Cristianini and J. Kandola,  On the concentration of spectral
properties. In: Dietterich, T.G., Becker, S., Ghahramani, Z. (eds.) Advances in
Neural Information Processing Systems 14, MIT Press, Cambridge,  pp. 511--517,  (2002).


\bibitem{Slepian1} D. Slepian and H. O. Pollak, Prolate spheroidal wave functions, Fourier analysis and
uncertainty I, {\it Bell System Tech. J.} {\bf 40} (1961), 43--64.

\bibitem{Slepian2} D. Slepian,  Prolate spheroidal wave functions, Fourier analysis and
uncertainty--IV: Extensions to many dimensions; generalized
prolate spheroidal functions, {\it Bell System Tech. J.} {\bf 43}
(1964), 3009--3057.


\bibitem{ST} R. Somaraju and J.  Trumpf,
Degrees of freedom of a communication channel: using DOF singular values.
IEEE Trans. Inform. Theory {\bf 56},  (2010) (4), 1560--1573.

\bibitem{Rosasco1} E. De Vito, L. Rosasco, A. Caponnetto, U. De Giovannini and F. Odone, Learning from examples as an inverse problem, 
{\it J. Mach. Learn. Res.} {\bf 6} (2005), 883--904.

\bibitem{BBZ2}  L. Zwald and G. Blanchard, On the Convergence of Eigenspaces in Kernel
Principal Component Analysis,   {\it Advances in Neural Information Processing Systems 18}, Proceedings of the 2005 Conference.

\end{thebibliography}
\end{document}